\documentclass[11pt]{article}
\usepackage{}
\usepackage{amssymb}
\usepackage{color, dsfont}
\usepackage{amsmath}
\usepackage{amsthm}
   
 \DeclareMathOperator*{\argmin}{arg\,min} 
\usepackage{bbm}
\usepackage{amsfonts}
\usepackage{mathrsfs}
\usepackage{dsfont}
\usepackage{extarrows}
\usepackage{algorithm}
\usepackage{algorithmic}
\usepackage{diagbox}
\usepackage{graphicx}
\usepackage{makecell}
\usepackage[numbers,sort&compress]{natbib}
\usepackage{setspace} 
\allowdisplaybreaks

\def\PP{\mathbb{P}}
\def\EE{\mathbb{E}}
\def\lt{\left}
\def\rt{ \right}

\newcommand{\mathsym}[1]{{}}

\def\leq{\leqslant}
\def\geq{\geqslant}

\newtheorem{thm}{Theorem}[section]
\newtheorem{prop}{Proposition}[section]
\newtheorem{lem}{Lemma}[section]
\newtheorem{rem}{Remark}[section]
\newtheorem{cor}{Corollary}[section]
\newtheorem{defn}{Definition}[section]
\newtheorem{exm}{Example}[section]
 
\newtheorem{ass}{Assumption}[section]

\numberwithin{equation}{section} \allowdisplaybreaks[4]
\newcommand{\be}{\begin{equation}}
\newcommand{\ee}{\end{equation}}
\newcommand{\ba}{\begin{eqnarray*}}
\newcommand{\ea}{\end{eqnarray*}}

\def\lb{\label}

\def\d{\mathrm{d}}

\begin{document}
\date{}
\pagestyle{plain}
\title{Algorithms for  zero-sum stochastic games with the  risk-sensitive average   criterion \footnote{This work was   supported by the National Key Research and Development Program of China (2022YFA1004600)}}
\author{Fang Chen\footnote{  Beijing International Center for Mathematical Research, Peking University, Beijing 100871, China. Email: chenfang@pku.edu.cn}, Xianping Guo\footnote{Corresponding author. School of Mathematics, Sun Yat-Sen University, Guangzhou 510275, China. Email: mcsgxp@mail.sysu.edu.cn},
Xin Guo \footnote{School of Science, Sun Yat-Sen University, Guangzhou 510275, China. Email: guox87@mail.sysu.edu.cn},
Junyu Zhang\footnote{School of Mathematics, Sun Yat-Sen University, Guangzhou 510275, China. Email:   mcszhjy@mail.sysu.edu.cn }}
\date{}
\maketitle \underline{}
	
{\bf Abstract:} This paper is an attempt to compute the value and   saddle points of   zero-sum risk-sensitive average stochastic games. For the average games with finite states and actions,  we first introduce the so-called  irreducibility coefficient and then establish its equivalence to the irreducibility condition.  Using this   equivalence,   we develop an iteration algorithm to  compute  $\varepsilon$-approximations  of the value  (for any given $\varepsilon>0$) and show its convergence. Based on   $\varepsilon$-approximations of the value and the   irreducibility coefficient, we further propose another iteration algorithm,
which is proved to obtain $\varepsilon$-saddle points in   finite  steps.
 Finally,  a  numerical example of energy management in smart grids is provided to   illustrate our results.

	
\setlength{\baselineskip}{0.25in}
		
\section{Introduction}

Stochastic games, introduced by Shapley  \cite{S53},  are an important class of stochastic optimality models and have been widely studied  due to  their rich applications  \cite{JN18-1,JN18-2}. Traditional  stochastic game theory  primarily focuses on   the expected  criteria \cite{GH11,GY08,  PL15,WC16}, where  all players are assumed to be risk-neutral.
However,  it is well known that  players may be risk-sensitive in  real-world applications \cite{FS00}.  Therefore,  risk-sensitive stochastic games,  in which the risk-sensitivity and reward of players are considered simultaneously through the exponential utility function  and   risk-sensitive  parameters, have garnered sustained research interest   \citep{BG14,BR17,CH19, GGPP23, BS23, CG25, BG18,GSPP21, HM21,W18}.

In  this paper, we focus on  zero-sum  discrete-time stochastic games with the risk-sensitive average  criterion, and thus   only describe the existing works on this aspect.
Basu  and Ghosh \cite{BG14} established the existence of  saddle points for  risk-sensitive average discrete-time stochastic games with countable states  and bounded costs. This result was extended to   the case of Borel state and action spaces by B{\"a}uerle  and Rieder  \cite{BR17}. 
For  risk-sensitive average discrete-time stochastic games with   unbounded costs,  Ghosh et al. \cite{GGPP23} employed the nonlinear Krein-Rutman theorem to  prove  the existence of  saddle points.   
 Cavazos-Cadena and Hern\'{a}ndez-Hern\'{a}ndez  \cite{CH19} showed that an appropriate normalization of   risk-sensitive discounted value  functions converges to the risk-sensitive average value function as the discount factor tends to $1$ in   discrete-time stochastic games.
 For   semi-Markov games, Bhabak and Saha  \cite{BS23} addressed   the  case of finite states and proved the existence of the value and   saddle points.
  Chen and Guo \cite{CG25} studied semi-Markov games with compact state and action  spaces, proved that the value and a saddle point  exist, and provided an algorithm to compute $\varepsilon$-approximation of the value.  
  
    As can be seen in  \cite{BG14,BR17,CH19,GGPP23,BS23,CG25},    for the  risk-sensitive average  games  the existence   of  saddle points has been established under various conditions, whereas their computation   has not been addressed even for the finite-state case. Since it is   desirable and important in practical applications to compute   saddle points, in this paper we will make an attempt on the computation. Precisely, we study a discrete-time risk-sensitive average stochastic game  with finite states and actions, and aim to provide algorithms for computing (or at least approximating)  the value and saddle points.

First,   we  introduce  the so-called irreducible coefficient,  give its efficient calculation from the data of the game, and establish the  equivalent relationship between the irreducibility condition and the irreducible coefficient (see Proposition \ref{prop-N}).  Second,  under the irreducibility condition  we construct  an increasing sequence and a decreasing one, whose limits are proved to equal the value. Based on these two sequences,  we propose an iteration algorithm and its stopping rule for computing $\varepsilon$-approximations of the value (see Algorithm \ref{algo1} and Theorem \ref{al-value}). Third, using    $\varepsilon$-approximations of the value and the irreducible coefficient, we further construct an approximate solution to the Shapley equation, and prove that any mini-max selector of the approximate solution is an $\varepsilon$-saddle point (see Theorem  \ref{thm-sd} and Proposition \ref{var-sd}).  Subsequently,    we present an algorithm for computing $\varepsilon$-saddle points, where the number of iteration steps is explicitly given (see Algorithm \ref{algo2}).  To the best of our knowledge, Algorithm  \ref{algo2}  developed in this paper is the first algorithm  to   compute    $\varepsilon$-saddle points of risk-sensitive average stochastic games. Finally, we give    an   example of energy management in smart grids  to illustrate the applications and effectiveness of our results.

The rest of this paper is organized as follows.  In Section 2, we describe the model of the stochastic game  and introduce the risk-sensitive average criterion.  Our main results on approximating the value and saddle points   are presented in Sections 3 and 4, respectively.  Finally, to  illustrate our results, we  give an     example   in Section 5.

 \section{The game model}	

A two-person zero-sum  risk-sensitive stochastic game   is  defined as  
	\begin{equation}\label{game}
G(\theta):=	\{ \theta, E, A, B, (A(i), B(i), i \in E), P(j|i, a, b), c(  i, a, b) \},
\end{equation}
 where $\theta\in (0,\infty)$ is a risk-sensitive parameter,  $E$ is the  state space, and $A$ and $B$ are the action  spaces  of players 1 and   2, respectively.  The spaces  $E$, $A$, and $B$  are assumed to be finite. For each $i\in E$, the nonempty set $A(i)\subset A$ (resp. $B(i)\subset B$) denotes the set of admissible actions of player 1 (resp. player 2) at state $i$.
 The stochastic kernel $P$ on $E$ given   $\mathbb{K}:=\{(i,a,b)| i\in E ,a\in A(i), b\in B(i)\}$ is the  transition function.
 Finally, the real-valued function $c$   on $  \mathbb{K}$ is the  cost/reward function to player 1/player 2. 
Since $E$, $A$, and $B$  are finite, $c$ is bounded. 
 Without loss of generality,     assume that   $c$ is nonnegative.

 Now we introduce the concept of a policy.  To do so, for each $n\in \mathbb{N}:=\{0,1,2,\ldots\}$, let $H_n$  denote the space  of admissible histories of the game up to the $n$th decision epoch, i.e.,  $H_0=E$,  $H_n=\mathbb{K}^n \times E (n\geq 1)$. 
\begin{defn}\lb{pol}
\begin{itemize}
	\item[\rm (a)]  A   (randomized history-dependent) policy  for player 1 is a sequence  $\pi = \{ \pi_n,  n \in \mathbb{N} \}$ of stochastic kernels $\pi_n$ on $A$ given $ H_n$  satisfying $\pi_n(A(i_n)|h_n)=1$ for all $n\in \mathbb{N}$ and $h_n=(i_0,a_0,b_0,\ldots,i_{n-1}, a_{n-1},b_{n-1},i_n)\in H_n$.

\item [\rm (b)] A  policy  $\pi = \{ \pi_n,  n \in \mathbb{N} \}$ for player 1 is called Markov if for each $n\in \mathbb{N}$ there exists a stochastic kernel   $\varphi_n$ on $A$ given $E$   such that $\pi_n(\cdot|h_n)=\varphi_n(\cdot|i_n)$ for all   $h_n=(i_0,a_0,b_0,\ldots,i_{n-1}, a_{n-1},b_{n-1},i_n)\in H_n$. We write such a policy as $\pi =\{\varphi_n,n\in \mathbb{N}\}$.
\item[\rm (c)]A Markov policy $\pi=\{\varphi_n,n\in \mathbb{N}\}$  for player 1 is said to be stationary if $\varphi_n$ are independent of $n$. In this case, we write $\pi$ as $\varphi$ for simplicity.
\item [\rm (d)] A stationary policy $\varphi$  for player 1 is called deterministic  stationary  if there is a map $f: E\to A$  satisfying $\varphi(\cdot|i)=\delta_{f(i)}(\cdot)$   for all $i\in E$, where $\delta_x$ denotes the Dirac measure at   $x$.  We denote such a deterministic  stationary policy   by $f$ for simplicity.
\end{itemize}
 \end{defn}
Denote by $\Pi_1$, $\Pi_1^m$, $\Pi_1^s$, and $\Pi_1^{sd}$   the sets of all  randomized history-dependent, Markov, stationary, and deterministic  stationary  policies   for player 1, respectively.
 Moreover, the sets $\Pi_2$, $\Pi_2^m$,  $\Pi_2^s$,   and $\Pi_2^{sd}$ of all  randomized history-dependent,  Markov,  stationary, and deterministic  stationary  policies, respectively,  for player 2 are defined similarly, with $B$ and $B(i)$ in lieu of $A$ and $A(i)$, respectively.

Given any initial state    $i \in E$, $\pi=\{\pi_n,n\in \mathbb{N}\}\in \Pi_1$,  and  $\sigma=\{\sigma_n,n\in \mathbb{N}\}\in \Pi_2$,
by   the Ionescu Tulcea theorem  \cite[Proposition C.10]{HL96},  there is a unique probability measure $\PP_{  i}^{\pi, \sigma}$ on $(( E \times A \times B  )^{\infty}, \mathcal{B}(( E \times A \times B  )^{\infty}))$  satisfying that 
\begin{align*}
&\mathbb{P}_i^{\pi,\sigma}(X_0=i)=1\\
&\mathbb{P}_i^{\pi,\sigma}(A_n=a,B_n=b|Y_n=h_n)=\pi_n(a|h_n)\sigma_n(b|h_n)  \quad \forall a\in A, b\in B,\\
&\mathbb{P}_i^{\pi,\sigma}(X_{n+1}=j|Y_n=h_n,A_n=a_n, B_n=b_n)=P(j|i_n,a_n,b_n)   \quad \forall j\in E,
\end{align*}
for all $n\in \mathbb{N}$, $h_n=(i_0,a_0,b_0,\ldots,i_{n-1}, a_{n-1},b_{n-1},i_n)\in H_n$, $a_n\in A(i_n)$, and $b_n\in B(i_n)$,  where for each $n\in \mathbb{N}$ and $\omega= (   i_0, a_0, b_0,  \ldots, i_n, a_n, b_n,   \ldots) \in ( E \times A \times B  )^{\infty}$,
\begin{align}\label{Yn}
	X_n (\omega) := i_n,  \   A_n (\omega) := a_n,  \   B_n (\omega) := b_n,\  Y_n(\omega)=(i_0,a_0,b_0,i_1,\ldots,a_{n-1},b_{n-1},i_n).  
\end{align} 
Let $\EE_{  i}^{\pi,\sigma}$ denote  the expectation operator with respect to $\PP_{ i}^{\pi,\sigma}$. Then, the   risk-sensitive    average cost for player 1  is defined by
\begin{align} \label{d-rs}
J(  i,\pi,\sigma):= \limsup_{n\to \infty}\frac{1}{\theta n}\ln\EE_{  i}^{\pi,\sigma} \lt[e^{\theta \sum_{k=0}^{n-1}  c( X_k,A_k,B_k)} \rt],
\quad i\in E,  (\pi,\sigma)\in \Pi_1\times \Pi_2. 
\end{align}
  The upper and lower value functions  are given by
\begin{align*}
 \overline{J}(i):=\inf_{\pi\in \Pi_1}\sup_{\sigma\in \Pi_2}J( i,\pi,\sigma), \quad\text{and}\quad \underline{J}(   i):=\sup_{\sigma\in \Pi_2}\inf_{\pi\in \Pi_1}J(i,\pi,\sigma), \quad i\in E,
\end{align*}
 respectively.
Moreover,   if $\underline{J}(i)=\overline{J}(i)$ for all $i\in E$,  the common function is called the value   of the  game  and is denoted by $J^*(i)$.	
\begin{defn}\label{defn-saddle}
Suppose that the value of the game exists and	fix any $\varepsilon\geq 0$.
\begin{itemize}
\item [\rm (a)]
 A policy $\pi^\varepsilon\in \Pi_1$ for player 1 is called $\varepsilon$-optimal   if
\begin{align*}
 J( i,\pi^\varepsilon,\sigma)\leq  {J}^*(i)+\varepsilon \quad \forall i\in E,\sigma\in \Pi_2.
\end{align*}
\item [\rm (b)] A policy $\sigma^\varepsilon\in \Pi_2$ for player 2 is called $\varepsilon$-optimal if
\begin{align*}
J( i,\pi,\sigma^\varepsilon) \geq   {J}^*(i)-\varepsilon\quad \forall i\in E, \pi \in \Pi_1.
\end{align*}
\item [\rm (c)] If  both $\pi^\varepsilon\in \Pi_1$ and $\sigma^\varepsilon\in \Pi_2$ are $\varepsilon$-optimal,   then the policy pair $(\pi^\varepsilon,\sigma^\varepsilon)$ is said to be an $\varepsilon$-saddle point. In particular, a $0$-saddle point is called a saddle point.
\item  [\rm (d)]A constant $r\in \mathbb{R}$ is called an  $\varepsilon$-approximation of the value if $\max_{i\in E}|J^*(i)-r|\leq \varepsilon$. 
 
\end{itemize}

\end{defn}

The aim of this paper is to develop   algorithms for computing (at least approximating) the value and  saddle points.

\section{On the computation of  $\varepsilon$-approximation of the value}

In this section,  we provide an algorithm  to  approximate   the value. 
To   this end,
we  recall the Shapley equation for the risk-sensitive average stochastic game, which is given by
\begin{align}\lb{SE}
	 \lambda h(i)=\inf_{\mu\in \mathcal{P}(A(i))}\sup_{\nu\in \mathcal{P}(B(i))} \sum_{a\in A(i)}\sum_{b\in B(i)}\sum_{j\in E}e^{\theta c(i,a,b)}h(j)P(j|i,a,b)\mu(a)\nu(b)   \ \forall  i\in E,
\end{align}
where  $\mathcal{P}(A(i))$  and  $\mathcal{P}(B(i))$  denote  the sets of  probability measures on $A(i)$ and $B(i)$, respectively,  $\lambda\in \mathbb{R}$,   $h$ is a function defined on $E$, and the pair $(\lambda,h)$ is called a solution to the Shapley equation (\ref{SE}).

Next, we  prove that the value exists by a solution to the Shapley equation. For convenience, we  introduce some  notation.   
 Let $\mathcal{M}$  be  the space  of all real-valued    functions $h$ on $E$ with the norm $\|h\|:=\max_{i\in E}|h(i)|$.
Note that $(\mathcal{M}, \|\cdot\|)$ is an ordered Banach space where $h\geq g$ means that $h(i)\geq g(i)$ for all $i\in E$. Denote
\begin{equation}\label{M}
	\mathcal{M}_+:=\{h\in \mathcal{M}| h(i)\geq  0, i\in E\}, \quad \mathcal{M}_+^o:=\{h\in \mathcal{M}| h(i)> 0, i\in E\}.
\end{equation}
Clearly,
$   \mathcal{M}_+^o \subset \mathcal{M}_+\subset \mathcal{M}$.
Now, we define an operator $\mathcal{L}$ on $\mathcal{M}_+$ as
\begin{equation}\label{L}
	\mathcal{L}h(i):=\inf_{\mu\in \mathcal{P}(A(i))}\sup_{\nu\in \mathcal{P}(B(i))}\mathcal{L}^{\mu,\nu}h(i), \quad h\in \mathcal{M}_+,i\in E,
\end{equation}
where
\begin{equation}\label{L-uv}
	\mathcal{L}^{\mu,\nu}h(i):=\sum_{a\in A(i)}\sum_{b\in B(i)}\sum_{j\in E}e^{\theta c(i,a,b)}h(j)P(j |i,a,b)\mu(a)\nu(b).
\end{equation}
It follows from (\ref{L}) and (\ref{L-uv}) that $\mathcal{L}$ maps $\mathcal{M}_+$ into  itself. Therefore, for each $n\geq 0$, we can define an operator $\mathcal{L}^n$ on $\mathcal{M}_+$ by
  $\mathcal{L}^0h=h$ and $\mathcal{L}^nh=\mathcal{L}(\mathcal{L}^{n-1}h)$ for each $h\in \mathcal{M}_+$.

\begin{defn}\label{max-min}
	Given any function $h\in\mathcal{M}_+$, a  stationary policy pair $(\varphi,\psi)\in \Pi_1^s\times \Pi_2^s$ is called a mini-max selector of $h$ if
	\begin{equation}\label{max-min-ineq}
		\mathcal{L}h(i)=\sup_{\nu\in \mathcal{P}(B(i))}\mathcal{L}^{\varphi(\cdot|i),\nu}h(i)=\inf_{\mu\in \mathcal{P}(A(i))}\mathcal{L}^{\mu,\psi(\cdot|i)}h(i)=\mathcal{L}^{\varphi(\cdot|i),\psi(\cdot|i)}h(i) \  \forall i\in E.
	\end{equation}
\end{defn} 
Since     $E$,  $A$, and $B$ are finite, the following result directly follow from    \cite[Section 2.4]{B13}.

\begin{lem}\label{lem-minimax}
 For any $i\in E$ and $h\in \mathcal{M}_+$, 
 the following linear program (Primal LP) 
 and its dual  program   (Dual LP) admit    optimal solutions,  denoted  by $(w_h(i),\varphi_h(\cdot|i))$ and $(v_h(i),\psi_h(\cdot|i))$, respectively. 
 \begin{align}\label{LP}
 	\begin{aligned}
 		& \textup{Primal LP:}\\
 		&\min_{w,\mu}  \, w   \\
 		&\mbox{s.t.} \left\{
 		\begin{array}{ll}
 			w\geq  \mathcal{L}^{\mu,\delta_{b}}h(i) &  \forall b\in B(i), \\
 			\sum_{a\in A(i)}\mu(a)=1,\\
 			\mu(a)\geq 0   &\forall a\in A(i);
 		\end{array}
 		\right.
 	\end{aligned}  \qquad  
	\begin{aligned}
			& \textup{Dual LP:}\\
 		&\max_{v,\nu}  \, v   \\
 		&\mbox{s.t.} \left\{
 		\begin{array}{ll}
 			v\leq  \mathcal{L}^{\delta_a,\nu}h(i) &  \forall a\in A(i), \\
 			\sum_{b\in B(i)}\nu(b)=1,\\
 			\nu(b)\geq 0   &\forall b\in B(i).
 		\end{array}
 		\right.
 	\end{aligned}
 \end{align}
Moreover,   $\mathcal{L}h(i)=w_h(i)=v_h(i)  $ for all $i\in E$ and $h\in \mathcal{M}_+$, and $(\varphi_h,\psi_h)$ is  a   mini-max selector of $h$.
\end{lem}
 
 To ensure the existence of the value and a saddle point,    we require  the  following irreducibility condition, which is   commonly used for   risk-sensitive average  game    \cite{BG14, CH19, GGPP23}. 
\begin{ass}[Irreducibility condition]\label{ir-ass}
	Given any   $(f,g)\in \Pi_1^{sd}\times \Pi_2^{sd}$,  the Markov chain	$\{X_n,n\in \mathbb{N} \}$ is  irreducible under the policy pair  $(f,g)$.
\end{ass}

 \begin{prop}\label{pro-value}
 	Under Assumption \ref{ir-ass},   the following statements hold.
 	\begin{itemize}
 			\item [\rm (a)]  The limit
 			 $
 						\lambda_{*}:=\lim\limits_{n\to \infty}\|\mathcal{L}^n {\bf 1}\|^{\frac{1}{n}}
 				$   exists, 	where ${\bf 1} $ represents the constant function one. 
 			\item [\rm (b)]  There exists a function $h^*\in \mathcal{M}_+^o  $ such that  $\lambda_* h^*=\mathcal{L}h^*$, i.e., the Shapley equation (\ref{SE}) has a solution $(\lambda_*, h^*)$.
 		      	\item [\rm (c)] The value of the game  exists and satisfies 
 			 	$ J^* (i)= {\theta}^{-1}{\ln \lambda_*}$  for all $i\in E.$    
 			Moreover, the game admits a stationary  saddle point $(\varphi^*,\psi^*)\in \Pi_1^s\times \Pi_2^s$.
 		\end{itemize}
 \end{prop}
\begin{proof}
(a)  From (\ref{L}) and (\ref{L-uv}), we  derive    
\begin{align}\label{L-mono}
\mathcal{L}h_1\geq \mathcal{L}h_2 \quad  \text{and}\quad \mathcal{L}rh_1= r\mathcal{L}h_1  \quad \forall h_1, h_2\in \mathcal{M}_+, h_1\geq h_2, r\geq 0.
\end{align}
This,  by  \cite[Lemma 2.0.7]{OT95}, yields   that
the limit $\lim_{n\to \infty} (\sup_{h\in \mathcal{M}_+,\|h\|\leq 1}\|\mathcal{L}^n h\| )^{\frac{1}{n}}$ exists. Furthermore, using an induction argument,    we obtain that
$ \mathcal{L}^n h \leq \mathcal{L}^n {\bf 1}$ for all $n\in \mathbb{N}$   and $h\in \mathcal{M}_+$ with $\|h\|\leq 1$, which implies   $\sup_{h\in \mathcal{M}_+,\|h\|\leq 1}\|\mathcal{L}^n h\|=
\|\mathcal{L}^n {\bf 1}\|$. Therefore, the limit $\lim_{n\to \infty}\|\mathcal{L}^n {\bf 1}\|^{\frac{1}{n}}$ exists and 
\begin{align}\label{Ln}
	\lambda_{*}=\lim_{n\to \infty}\|\mathcal{L}^n {\bf 1}\|^{\frac{1}{n}}=\lim_{n\to \infty} (\sup_{h\in \mathcal{M}_+,\|h\|\leq 1}\|\mathcal{L}^n h\| )^{\frac{1}{n}}.
\end{align}
 (b) Since  $c$ is nonnegative,  we have $\mathcal{L}{\bf 1}\geq {\bf 1}$. Thus, by the monotonicity
of $\mathcal{L}$,      we obtain 
\begin{equation}\label{L1-ineq}
	\mathcal{L}^{n+1} {\bf 1} \geq \mathcal{L}^{n } {\bf 1}  \geq {\bf 1} \quad \forall n \in \mathbb{N}.
\end{equation}
This, together with  (\ref{Ln}),   implies  $ \lambda_{*} = \lim_{n\to \infty}\|\mathcal{L}^n {\bf 1}\|^{\frac{1}{n}}\geq 1.$ 
On the other hand,	by   (\ref{L-uv}) and  (\ref{L})   we derive  that
	\begin{equation}\label{L-1}
			\|\mathcal{L}u -\mathcal{L}v \| =\max_{i\in E}| \mathcal{L}u(i) -\mathcal{L}v(i) | \leq  e^{\theta ||c||}\|u-v\| \quad \forall u,v\in \mathcal{M}_+,
		\end{equation}
with $ ||c||:=\max_{(i,a,b)\in \mathbb{K}}c(i,a,b) $, which implies that   the set $\{\mathcal{L}u| u\in\mathcal{M}_+, \|u\|\leq r\}$ is bounded for any $r\geq 0$. Then, by   the finiteness of $E$, we have that  the set $\{\mathcal{L}u| u\in \mathcal{M}_+, \|u\|\leq r\}$ is relatively compact, which together with 	
	  (\ref{L-1}) implies that  $\mathcal{L}$ is  a compact operator.
	According to  the nonlinear Krein-Rutman theorem \cite[Proposition 3.1.5]{OT95}, 
	 the compactness of $\mathcal{L}$, (\ref{L-mono}), and the inequality $\lambda_*\geq 1$ guarantee that
	 there exists a function $h^*\in \mathcal{M}_+ $  satisfying $\lambda_{*}h^*=\mathcal{L}h^*$ and $||h^*||\neq 0$.
	
	Next, we prove   $h^*\in \mathcal{M}_+^o$. 
	By Lemma \ref{lem-minimax}, there exists  $(\varphi^*,\psi^*)\in \Pi_1^s\times \Pi_2^s$ such that  $\mathcal{L}h^*(i)=\mathcal{L}^{\varphi^*(\cdot|i),\psi^*(\cdot|i)}h^*(i)$ for all $i\in E$. This, together with the nonnegativity of $c$ and (\ref{L-uv}), yields that
$
\mathcal{L}h^*(i) \geq  \sum_{a\in A(i)}\sum_{b\in B(i)}\sum_{j\in E}h^*(j) P(j|i,a,b)\varphi^*(a|i) \psi^*(b|i)$  for all $i\in E$.
 Then, using the equality $\lambda_{*}h^*=\mathcal{L}h^*$,  we can prove by induction that
		\begin{equation}\label{L-ineq2}
				h ^* (i)\geq 	
			 \lambda_*^{-n}  \mathbb{E}^{\varphi^* ,\psi^* }_i[h^*(X_{n }) ] \quad \forall n\geq 0,i\in E.
			\end{equation}
 From $h^*\in \mathcal{M}_+ $ and $||h^*||\neq 0$, there is a state $j\in E$ with $h^*(j)>0$. 	On the other hand, under Assumption \ref{ir-ass}, for each $i\in E$,  there is a positive integer $N_i$ such that
		$\mathbb{P}^{\varphi^* ,\psi^*}_{{i} }(X_{N_i}=j)>0$. Therefore,   using   (\ref{L-ineq2}) we obtain that
	$$
			h^*(i )\geq    \lambda_*^{-N_i}     \mathbb{E}^{\varphi^* ,\psi^*}_{ {i} }[ h^*(X_{N_i})]\geq  \lambda_*^{-N_i}  h^* (j)   \mathbb{P}^{\varphi^* ,\psi^*}_{{i} }(X_{N_i}=j)>0 \quad \forall i\in E,
		$$
which implies $h^*\in \mathcal{M}_+^o$. 
		
	 (c)  The desired result  follows by combining \cite[Theorem 1]{BS23}  and part (b).
\end{proof}
Proposition \ref{pro-value} shows that  under the irreducibility condition the game  has the value and a saddle point. Next, we  introduce the irreducibility  coefficient  and establish    the relationship between  it and the irreducibility  condition, which is used to   propose an   algorithm for computing $\varepsilon$-approximations of the value.

\begin{defn}\label{defn-ir-coeff}
The	 irreducible coefficient is defined as
\begin{align}\label{gamma}
	\gamma:=\min_{(i,j)\in E^2} \inf_{\pi\in \Pi_1}\inf_{\sigma\in \Pi_2}\mathbb{P}_i^{\pi,\sigma}(\tau_j\leq |E|),
\end{align} 
where   $\tau_j:=\inf\{n>0 \mid X_n=j\}$ denotes the first return time  to  $j\in E $ with $\inf \emptyset:=+\infty$, and   $|E|$ is the cardinality of the state space $E$.
\end{defn}

\begin{prop}\label{prop-N}  
\begin{itemize}
	
\item [\rm (a)] For each $i,j\in E$, define  
\begin{align}\label{Vxy}
  	V_0^j(i):=1 \quad \text{and} \quad	V_{k+1}^j(i ):=\max_{a\in   A(i)}\max_{b \in   B(i)}\sum_{i_1\neq j}P(i_1 |i,a,b)V_k^j(i_1 ), \quad  k\in \mathbb{N}.
\end{align}
Then, the  irreducible coefficient $\gamma$ satisfies $\gamma=1-\max_{(i,j)\in E^2}V_{|E|}^j(i) $.
\item [\rm (b)] Assumption \ref{ir-ass} holds if and only if   $\gamma>0$.
\end{itemize}
\end{prop}
\begin{proof}
(a)  Given any  $k\geq 1$,  $\pi=\{\pi_n, n\geq 0\}\in \Pi_1$,    $\sigma=\{\sigma_n,n\geq 0\}\in \Pi_2$,   and   $i,j\in E$, it follows from (\ref{Vxy}) that
\begin{align}
	&\mathbb{P}_i^{\pi,\sigma}(\tau_j\leq k)\notag\\
	=&1- \sum_{i_0\in E}\delta_{i}(i_0)\sum_{a_0\in A(i_0)}\sum_{b_0\in B(i_0)}\sum_{i_1\in E\setminus\{j\}}P(i_1|i_0,a_0,b_0)\pi_0(a_0|h_0)\sigma_0(b_0|h_0)\cdots\sum_{a_{k-1}\in A(i_{k-1})} \notag\\
	&\sum_{b_{k-1}\in B(i_{k-1})}\sum_{i_{k }\in E\setminus\{j\}} P(i_{k }|i_{k-1},a_{k-1},b_{k-1})\pi_{k-1}(a_{k-1}| h_{k-1})\sigma_{k-1}(b_{k-1}|h_{k-1})\notag\\
	\geq &1- \sum_{i_0\in E}\delta_{i}(i_0)\sum_{a_0\in A(i_0)}\sum_{b_0\in B(i_0)}\sum_{i_1\in E\setminus\{j\}}P(i_1|i_0,a_0,b_0)\pi_0(a_0|h_0)\sigma_0(b_0|h_0)\cdots\sum_{a_{k-2}\in A(i_{k-2})} \notag\\
	&\sum_{b_{k-2}\in B(i_{k-2})}\sum_{i_{k-1 }\in E\setminus\{j\}} V_1^j(i_{k-1})P(i_{k -1}|i_{k-2},a_{k-2},b_{k-2})\pi_{k-2}(a_{k-2}| h_{k-2})\sigma_{k-2}(b_{k-2}|h_{k-2}) \notag\\
	\vdots \notag\\
	\geq& 1 - \sum_{i_0\in E}\delta_{i}(i_0)\sum_{a_0\in A(i_0)}\sum_{b_0\in B(i_0)}\sum_{i_1\in E\setminus\{j\}}V_{k-1}^j(i_1)P(i_1|i_0,a_0,b_0)\pi_0(a_0|h_0)\sigma_0(b_0|h_0) \notag\\
	\geq& 1-V_{k }^j(i), \label{VK-eq2}
\end{align}
 where $h_0=i_0$ and $h_m=(h_{m-1},a_{m-1},b_{m-1},i_m)$ for all $m=1,2,\ldots,k$.  Hence,  
\begin{align}\label{VE-eq}
\inf_{\pi\in\Pi_1}\inf_{\sigma\in \Pi_2}\mathbb{P}_i^{\pi,\sigma}(\tau_j\leq  k) \geq 1-V_{ k}^j(i) \quad \forall i,j\in E.
\end{align}
Next, we claim that  for all $k\in \mathbb{N}$ and $j\in E$, there is $(\pi^{j,k},\sigma^{j,k})\in \Pi_1^m\times \Pi_2^m$ satisfying
\begin{equation}\label{VK-eq}
	 	V_{k}^j(i )=\mathbb{P}_i^{\pi^{j,k},\sigma^{j,k}}(\tau_j > k)\quad \forall i\in E.
\end{equation}
Fix any $j\in E$.  The proof  of the claim proceeds by induction on $k$. For $k=0$, 
 it follows from  $\{\tau_j> 0\}=(E\times A\times B)^\infty$ that  (\ref{VK-eq}) holds  with any Markov policy pair.   
Now assume  that   there exist $\pi^{j,k}=\{\varphi_0[k], \varphi_1[k],\ldots   \} \in \Pi_1^m$ and
$ \sigma^{j,k}= \{\psi_0[k], \psi_1[k],\ldots \}  \in  \Pi_2^m$ such that (\ref{VK-eq}) holds for some $k\geq 0$.
For each $i\in E$, by the finiteness of $A(i)$ and $B(i)$,   there exists  $(a^i,b^i)\in A(i)\times B(i)$ such that  $V_{k+1}^j(i )=\sum_{i_1\neq j}
P(i_1|i,a^i,b^i)V_{k}^j(i_1 ) $.
Then, we  construct $\pi^{j,k+1}:=\{\varphi_0[k+1],\ldots,\varphi_n[k+1],\ldots\}$ and $\sigma^{j,k+1}:=\{\psi_0[k+1],\ldots,\psi_n[k+1],\ldots\}$ as follows:   for each $i\in E$ and $n\geq 0$,
\begin{equation}\label{pi}
	\varphi_n[k+1](\cdot|i)=\begin{cases}
		\delta_{a^i}(\cdot)  \quad &n=0;\\
		\varphi_{n-1}[k](\cdot|i)\quad &n\geq 1;
	\end{cases}
	\,
	\psi_n[k+1](\cdot|i)=\begin{cases}
		\delta_{b^i}(\cdot)  \quad &n=0;\\
		\psi_{n-1}[k](\cdot|i)\quad &n\geq 1.
	\end{cases}
\end{equation}
The constructions of $\pi^{j,k+1}$ and $\sigma^{j,k+1}$ indicate that
\begin{align}\label{Vk-ineq1}
	&	\mathbb{P}^{\pi^{j,k+1},\sigma^{j,k+1}}_i(\tau_j>k+1)\notag\\ = &  \sum_{i_1 \neq j}P(i_1|i,a^i,b^i)\mathbb{P}^{\pi^{j,k},\sigma^{j,k}}_{i_1}(\tau_j>k)
	= \sum_{i_1 \neq j}P(i_1|i,a^i,b^i)V_k^j(i_1 )
	=  V_{k+1}^j(i )   \quad\forall i\in E,
\end{align}
where the first, second, and   last equalities follow from the Markov property, our induction hypothesis, and the definition of $(a^i,b^i)$, respectively. Clearly, (\ref{Vk-ineq1}) means that (\ref{VK-eq}) holds for integer $k+1$ and state $j$. This completes the inductive proof of the claim. Combining   (\ref{VE-eq}) and (\ref{VK-eq}) implies that for all $i,j\in E$
\begin{align}\label{V-mono}
	1-	V_{k}^j( i ) =\mathbb{P}_i^{\pi^{j,k},\sigma^{j,k}}(\tau_j \leq k) \geq \inf_{\pi\in \Pi_1}\inf_{\sigma\in \Pi_2}\mathbb{P}_i^{\pi  ,\sigma }(\tau_j \leq k) \geq 	1-	V_{k}^j( i ).
\end{align}
Setting $k=|E|$ in the above display gives that
$$	1-	V_{|E|}^j( i ) =\inf_{\pi\in \Pi_1}\inf_{\sigma\in \Pi_2}\mathbb{P}_i^{\pi  ,\sigma }(\tau_j \leq |E|) \quad \forall i,j\in E,
$$
which, together with the arbitrariness of $i,j $ and Definition \ref{defn-ir-coeff}, yields  the desired result. 

(b)  \underline{Assumption \ref{ir-ass}~$\Longrightarrow$~ $\gamma>0$:}
  Let $j\in E$ be an arbitrary state and  define  
	\begin{equation}\label{Dy}
	D_0:=\emptyset, \quad	D_{k+1}:=\big\{i\in E| \min_{a\in A(i), b\in B(i)}P(D_k\cup\{j\}|i,a,b)>0\big\},\quad k\in \mathbb{N}.
	\end{equation}
Since $D_0=\emptyset \subset D_1$, an induction argument shows that
$D_k\subset D_{k+1}$ for all $k\in \mathbb{N}$. Therefore, it follows from the finiteness of $E$  that   $k^*:=\inf\{k\in \mathbb{N}| D_k=D_{k+1}\} \leq|E|$.   Next, we shall prove $\gamma>0$ in steps.

{\bf Step 1:}	We  first  prove  $D_{k^*}=E$.

Suppose  by way of contradiction that   $   E \setminus D_{k^*} \neq \emptyset$.   Define
$$
(f^*(i),g^*(i)):= \argmin_{a\in A(i), b\in B(i)}P(D_{k^*}\cup\{j\}|i,a,b) \quad \forall i\in E.
$$
Using the definition  of $k^*$ and (\ref{Dy})  we have that
\begin{equation}\label{PD_k}
	\mathbb{P}_i^{f^*,g^*}(X_1\in D_{k^*}\cup\{j\})=\min_{a\in A(i), b\in B(i)}P(D_{k^*}\cup\{j\}|i,a,b)=0 \quad \forall i\in  E\setminus D_{k^*}.
\end{equation}
Suppose that   $\mathbb{P}_i^{f^*,g^*}(X_n\in D_{k^*}\cup\{j\})=0$  holds for some $n\geq 1$ and all $i\in  E\setminus D_{k^*}$,  then  the Markov property and (\ref{PD_k})  lead to  that for each $i\in  E\setminus D_{k^*}$
\begin{align*}
	 \mathbb{P}_i^{f^*,g^*}(X_{n+1}\in D_{k^*}\cup\{j\})
	=&\sum_{i_n\in  E}\mathbb{P}_i^{f^*,g^*}(X_n=i_n)\mathbb{P}_{i_n}^{f^*,g^*}(X_{ 1}\in D_{k^*}\cup\{j\})\\
	=&\sum_{i_n\notin  D_{k^*}\cup\{j\}}\mathbb{P}_i^{f^*,g^*}(X_n=i_n) \mathbb{P}_{i_n}^{f^*,g^*}(X_{ 1}\in D_{k^*}\cup\{j\})
	 =  0.
\end{align*}
Therefore, by induction  we obtain that
$ \mathbb{P}_i^{f^*,g^*}(X_n\in D_{k^*}\cup\{j\})=0$
 for all $n\geq 1$ and $i\in  E\setminus D_{k^*}$.
   Consequently, for any $i\in  E\setminus D_{k^*}$,  we have that $ \mathbb{P}^{f^*,g^*}_i(X_n=j)=0$ for all $n\geq 1$, which contradicts   Assumption \ref{ir-ass}.  Thus,   $D_{k^*}=E$.

{\bf Step 2:}		Next, we claim that for each $k\geq 1$
	\begin{equation}\label{P-tau}
		\mathbb{P}_i^{\pi,\sigma}(\tau_j\leq k)>0 \quad \forall i\in D_k, (\pi,\sigma)\in \Pi_1^m\times \Pi_2^m.
	\end{equation}
	For $k= 1$,  (\ref{P-tau})  directly follows from (\ref{Dy}) and $\{\tau_j\leq 1\}=\{X_1=j\}$. We now assume that (\ref{P-tau}) holds for some $k\geq 1$. Fix   $\pi=\{\varphi_0,\varphi_1,\ldots\}\in \Pi_1^m$,   $\sigma=\{\psi_0,\psi_1,\ldots\}\in \Pi_2^m$ and $i\in D_{k+1}$.
	We  consider the following case $1$ and  case $2$.
\begin{itemize}	
	\item []
	{\it Case 1:  $\mathbb{P}^{\pi,\sigma}_i(X_1=j)>0$.} Then, $\mathbb{P}_i^{\pi,\sigma}(\tau_j\leq k+1)\geq \mathbb{P}^{\pi,\sigma}_i(X_1=j)>0.$
\item[]	
 {\it Case 2:  $\mathbb{P}^{\pi,\sigma}_i(X_1=j)=0$.} Noting that $i\in D_{k+1}$, we get that
	$$\mathbb{P}^{\pi,\sigma}_i(X_1\in D_k)=\mathbb{P}^{\pi,\sigma}_i (X_1\in D_k\cup \{j\})\geq \min_{a\in A(i), b\in B(i)}P(D_k\cup \{j\}|i,a,b)>0.$$
Therefore, there is a state $i_1\in D_k\setminus\{j\}$ such that $\mathbb{P}^{\pi,\sigma}_i(X_1=i_1)>0$. Then,  the induction hypothesis  and the Markov property imply that
	$ 
	\mathbb{P}_i^{\pi,\sigma}(\tau_j\leq k+1)\geq \mathbb{P}_i^{\pi,\sigma}(X_1=i_1)\mathbb{P}_{i_1}^{\pi^1,\sigma^1}(\tau_j\leq k)
	>0,
	$ 
	where $\pi^1:=\{\varphi_{1},\varphi_2,\ldots \}$ and $  \sigma^1:=\{\psi_{1},\psi_2,\ldots  \}$.
 \end{itemize}
Thus, using the results for the two cases above, we see that (\ref{P-tau}) is also true  for $k+1$. Hence, by induction (\ref{P-tau}) holds for all $k\geq 1$.
	
{\bf Step 3:}  Noting that $D_{k^*}=E$ (see Step 1) and   $k^*\leq |E|$, we obtain   $D_{|E|}=E$.   Thus, by (\ref{P-tau}),  (\ref{VK-eq}), and (\ref{VK-eq2}),  there is a policy pair $(\pi ^{j,|E|},\sigma^{j,|E|})\in \Pi_1^m\times \Pi_2^m$ such that
\begin{align}\label{a-imply-b}
 \inf_{\pi\in \Pi_1}\inf_{\sigma\in \Pi_2}\mathbb{P}^{\pi,\sigma}_i(\tau_j\leq |E|)\ =\mathbb{P}^{\pi ^{j,|E|},\sigma^{j,|E|}}_i(\tau_j\leq |E|)>0 \quad \forall i\in E,
 \end{align}
which together with the finiteness of $E$ and (\ref{gamma}) yields  $\gamma>0$.

\underline{$\gamma>0$ ~$\Longrightarrow$~ Assumption \ref{ir-ass}}:  It follows from (\ref{gamma})   that
$$
 \mathbb{P}^{f,g}_i(\tau_j\leq |E|)\geq \gamma >0 \quad \forall (f,g)\in \Pi_1^{sd}\times \Pi_2^{sd}, i,j\in E.
$$
  Therefore,  for any $(f,g)\in \Pi_1^{sd}\times \Pi_2^{sd}$ and $i,j\in E$, by $\gamma>0$ there exists a positive integer $k\leq |E|$  such that 	$ 
  \mathbb{P}^{f,g}_i(X_k=j) > \gamma/|E|>  0,$  which    implies Assumption \ref{ir-ass}.
\end{proof}

\begin{rem}\label{rem-ass}
   By  iteratively calculating  $V^j_{|E|}$  via (\ref{Vxy}), Proposition \ref{prop-N}  provides a novel method for the verification of  the irreducibility condition.  Specifically, if $\max\limits_{i,j\in E}V^j_{|E|}(i)<1$, then  we see that  the irreducibility condition is satisfied; Otherwise, it is not.

\end{rem}

 Now, we present the main result of this section, which is that under Assumption \ref{ir-ass}, an
$\varepsilon$-approximation of the value   is computed  for any error  $\varepsilon>0$.

 \begin{thm}\label{al-value}
 	 For each $n\in \mathbb{N}$, define
\begin{equation}\label{lambda_n}
\lambda_n:= (\max_{i\in E}\mathcal{L}^{2^n}{\bf 1}(i))^{\frac{1}{2^n}}, \quad
\zeta_n:=(\min_{i\in E}\mathcal{L}^{2^n}{\bf 1}(i))^{\frac{1}{2^n}}.
\end{equation}
 Under Assumption \ref{ir-ass}, the following statements are valid.
 \begin{itemize}
 \item [\rm (a)] The sequence $\{\theta^{-1}\ln\lambda_n, n\in \mathbb{N}\}$ is decreasing in $n$ and $\rho^{*}= \lim_{n\to\infty} \theta^{-1}\ln\lambda_n$, where $\rho^*$ is  the value of the game.
 \item [\rm (b)] The sequence $\{\theta^{-1}\ln\zeta_n, n\in \mathbb{N}\}$ is increasing in $n$ and $\rho^*=\lim_{n\to \infty} \theta^{-1}\ln\zeta_n$.
 \item [\rm (c)] Given any $\varepsilon>0$, there exists a positive integer $n_{\varepsilon}$ such that $  {\lambda_{n_{\varepsilon}}}/{\zeta_{n_{\varepsilon}}}  \leq  e^ {\theta\varepsilon }$. Therefore,
 $ 0\leq \theta^{-1}\ln \lambda_{n_{\varepsilon}}-\rho^* \leq \varepsilon$, and thus    $ \theta^{-1}\ln \lambda_{n_{\varepsilon}}$ is an $\varepsilon$-approximation of the value.
 \end{itemize}
 \end{thm}

\begin{proof}
(a) From (\ref{L-mono}), for each $n\in \mathbb{N}$, we obtain    
\begin{align}\label{L-lam1}
 \max_{i\in E}\mathcal{L}^{2^{n+1}}{\bf 1} (i) \leq (\max_{j\in E}\mathcal{L}^{2^n}{\bf 1}(j))\max_{i\in E}\mathcal{L}^{2^n}{\bf 1}(i)=  (\max_{i\in E}\mathcal{L}^{2^{n }}{\bf 1}(i))^2.
\end{align}
 Combining (\ref{L-lam1}) and (\ref{lambda_n})  gives  that
$
 \lambda_{n+1}\leq (\max_{i\in E}\mathcal{L}^{2^{n }}{\bf 1}(i))^{\frac{2}{2^{n+1}}}= \lambda_n  $ for all $n\in \mathbb{N}$. Therefore, the sequence $\{\theta^{-1}\ln\lambda_n, n\in \mathbb{N}\}$ is decreasing in $n$. 
Moreover,  (\ref{Ln}) and the definition of $\lambda_n$ give   $\lambda_{*}=\lim\limits_{n\to\infty}\lambda_n$. Therefore,  by $\rho^*=\theta^{-1}\ln \lambda_*$, we get (a).

(b) Using a similar argument as (\ref{L-lam1}), we have
\begin{equation}\label{zetan}
\zeta_{n+1}=(\min_{i\in E}\mathcal{L}^{2^{n+1}}{\bf 1}(i))^{\frac{1}{2^{n+1}}}\geq ((\min_{i\in E}\mathcal{L}^{2^{n}}{\bf 1}(i))^2)^{\frac{1}{2^{n+1}}}=\zeta_n  \quad \forall n\in \mathbb{N}.
\end{equation}
Thus, the sequence $\{\theta^{-1}\ln\zeta_n, n\in \mathbb{N}\}$ is increasing, and the limit $\zeta_*:=\lim_{n\to\infty}\zeta_n $ exists. Moreover,  (\ref{L1-ineq}) gives  $\zeta_*>0$.  Observing that  $\zeta_n \leq \lambda_n$, by part (a) we have $\zeta_*\leq \lambda_{*}$.
Hence, to get   (b),   it suffices to show  $\zeta_*\geq{\lambda_{*}}$.  Suppose,     by contradiction,  that $\zeta_*<{\lambda_{*}}$.  Then, by  Proposition \ref{pro-value} (a),
 there is a positive integer $N$  satisfying
\begin{equation}\label{gamma-1}
  \|\mathcal{L}^k {\bf 1}\|^{\frac{1}{k}} \geq \lambda_*-\frac{\lambda_*-\zeta_*}{2}= \frac{\lambda_*+\zeta_*}{2\zeta_*}\zeta_* \quad \forall k\geq N.
\end{equation}
Since   $   \frac{\lambda_*+\zeta_*}{2\zeta_*}>1$ and  $ \gamma>0$ (by Proposition \ref{prop-N} (b)),  there exists  $m\in \mathbb{N} $  satisfying
\begin{equation*}
N_m:= 2^{m }-|E|\geq N \quad \text{and}  \quad    \Big(\frac{\lambda_*+\zeta_*}{2\zeta_*}\Big)^{  N_m}\geq 2  \gamma^{-1} |E|   \zeta_*^{|E|}.
\end{equation*}
Let $j^*:=\arg\max_{j\in E}\mathcal{L}^{  N_m}{\bf 1}(j)$. Then,  the above display and (\ref{gamma-1}) imply that
\begin{equation}\label{gamma2}
\mathcal{L}^{  N_m}{\bf 1}(j^*)=\|\mathcal{L}^{N_m}{\bf 1}\|\geq  \Big(\frac{\lambda_*+\zeta_*}{2\zeta_*}\Big)^{N_m}\zeta_*^{N_m}   \geq 2\gamma^{-1} |E|     \zeta_*^{2^m}.
\end{equation}
For each  $n= 1,2,\ldots,2^m $, by Lemma \ref{lem-minimax}, there exists  $(\varphi_{n-1},\psi_{n-1})\in \Pi_1^s\times \Pi_2^s$ such that
\begin{align}\lb{n-varphi-psi}
	\mathcal{L}^{\varphi_{n-1}(\cdot|i),\psi_{n-1}(\cdot|i)}\big(\mathcal{L}^{2^{m }-n }{\bf 1}\big)(i)=\mathcal{L}^{2^{m }-n+1 }{\bf 1}(i) \quad \forall i\in E.
\end{align}
Define $ \widehat{\pi}:=\{\widehat{\varphi}_n,n\in \mathbb{N}\}$ and  $\widehat{\sigma}:=\{\widehat{\psi}_n, n\in \mathbb{N}\}$ as
\begin{equation}\label{hat-pi}
\widehat{\varphi}_n:=\begin{cases}
\varphi_n, \quad   n=0,1,\ldots,2^{m }-1,\\
\varphi_{0}, \quad n> 2^{m }-1;
\end{cases}
\quad
\widehat{\psi}_n:=\begin{cases}
	\psi_n, \quad   n =0,1,\ldots,2^{m }-1,\\
	\psi_{0}, \quad n> 2^{m }-1.
\end{cases}
\end{equation}
Then, employing   (\ref{hat-pi}) and  (\ref{n-varphi-psi}), we deduce that
\begin{equation}\label{L-eq1}
\mathcal{L}^{2^{m }}{\bf 1}(i)=\mathbb{E}^{\widehat{\pi},\widehat{\sigma}}_i\big[e^{\theta \sum_{k=0}^{n-1} c(X_k,A_k,B_k)}\mathcal{L}^{2^{m }-n }{\bf 1}(X_{n })  \big] \quad  \forall i\in E, n= 1,2,\ldots,2^{m }.
\end{equation}
Given any $i\in E$, since    Proposition \ref{prop-N}  shows  $\mathbb{P}_i^{\widehat{\pi},\widehat{\sigma}}(\tau_{j^*}\leq |E|)\geq \gamma>0$, there is an integer $1\leq k_i  \leq |E|$ such that $\mathbb{P}_i^{\widehat{\pi},\widehat{\sigma}}(X_{k_i} =j^*)\geq  {\gamma }/{|E|}$, which, together with (\ref{L-eq1}) and  the nonnegativity of $c$,  implies that
\begin{align*}
\mathcal{L}^{2^{m }}{\bf 1}(i)\geq \mathbb{E}^{\widehat{\pi},\widehat{\sigma}}_i\big[  \mathcal{L}^{2^{m }-k_i  }{\bf 1}(X_{k_i })   \big]
 \geq    \mathbb{P}_i^{\widehat{\pi},\widehat{\sigma}}(X_{k_i }=j^*) \mathcal{L}^{2^{m }-k_i  }{\bf 1}(j^*)    \geq    { \gamma\mathcal{L}^{2^{m }-k_i  }{\bf 1}(j^*)}/{|E|}.
\end{align*}
Noting that $1\leq k_i  \leq |E|$ for all $i\in E$, by (\ref{L1-ineq}),  (\ref{gamma2}),  and the display above we get that
$$
\mathcal{L}^{2^{m }}{\bf 1}(i)\geq \gamma\mathcal{L}^{2^{m }-k_i  }{\bf 1}(j^*)/|E| \geq  \gamma\mathcal{L}^{2^{m }-|E|  }{\bf 1}(j^*)/|E| \geq 2\zeta_*^{2^m} \quad \forall i\in E,
$$
 which, together with  $\zeta_* >0$ and (\ref{lambda_n}),   gives
$
	\zeta_m=(\min\limits_{i\in E}\mathcal{L}^{2^{m }}{\bf 1}(i))^{\frac{1}{2^m}}\geq 2^{\frac{1}{2^m}}  \zeta_*> \zeta_*.
$
This contradicts    (\ref{zetan}). Thus,  $\zeta_*\geq{\lambda_{*}}$.

(c) Part (c) directly follows from Theorem \ref{pro-value}, parts (a) and (b).
\end{proof}

Based on Theorem \ref{al-value} and Lemma \ref{lem-minimax},   we propose an iteration algorithm  to compute $\varepsilon$-approximations of the value.

\begin{algorithm}
	\caption{(An algorithm of $\varepsilon$-approximations of the value.)}
	\label{algo1}
	\begin{algorithmic}
\STATE {\bf Input: } The data of the game and an accuracy $\varepsilon>0$.
\STATE {\bf Step 1}~(Initialization):  Set $n=1$ and $h_{0}= \mathcal{L}{\bf 1}$.
\STATE {\bf Step 2}~(Iteration):  Compute $h_n=\mathcal{L}^{2^{n-1}}h_{n-1}$ via (\ref{LP}), and set
 $$\zeta_n=(\min\limits_{i\in E}h_n(i))^{\frac{1}{2^n}}, \quad    \quad \quad \lambda_n=(\max\limits_{i\in E}h_n(i))^{\frac{1}{2^n}}.$$
\STATE {\bf Step 3} (Accuracy control):  If $ \frac{\lambda_n}{\zeta_n}  \leq  e^ {\theta\varepsilon } $, go to Step 4. Otherwise, go to Step 2 by replacing $n$ with $n + 1$.
 \STATE {\bf Step 4} (Approximate the value): Let  $\widetilde{\rho} =\theta^{-1}\ln \lambda_n$,  thus $ 0\leq \widetilde{\rho} -\rho^* \leq \varepsilon$.
  \STATE {\bf Output:} An $\varepsilon$-approximation   $\widetilde{\rho} $ of the value.
	\end{algorithmic}
\end{algorithm}

\section{On the computation of  $\varepsilon$-saddle points}

In this section, we aim to  compute $\varepsilon$-saddle points. 
To this end, we need the following auxiliary proposition.
\begin{prop}\label{var-sd}
Suppose that Assumption \ref{ir-ass} holds.   Given any $\varepsilon \geq 0$, if a function $h$ on $E$ satisfies the condition $h\leq h^*\leq e^{\theta \varepsilon} h$ with $h^*$ as in Proposition \ref{pro-value} (b), then each  mini-max selector $(\varphi_h,\psi_h)$ of   $h$   is an $\varepsilon$-saddle point.
\end{prop}
\begin{proof}
 	By Proposition \ref{pro-value}, we have    $h^*=e^{-\theta \rho^*} \mathcal{L}h^*$. This, together with $h\leq h^*\leq e^{\theta \varepsilon} h$, (\ref{L-mono}),  and    Definition  \ref{max-min},  implies 
		\begin{equation}\label{L-ineq1}
				h^*(i) \leq   {e^{ \theta (\varepsilon-\rho^*)}}   \mathcal{L}h (i)  \leq  {e^{ \theta (\varepsilon-\rho^*)}}  \mathcal{L}^{\mu,\psi_h(\cdot|i)}h (i) \leq  {e^{ \theta (\varepsilon-\rho^*)}}   \mathcal{L}^{\mu,\psi_h(\cdot|i)} h^* (i),
			\end{equation}
		for all $ i\in E$ and  $\mu\in  \mathscr{P}(A(i))$. Fix any   $\pi=\{\pi_0,\pi_1,\ldots\}\in \Pi_1$. Then by  (\ref{L-uv}) and (\ref{L-ineq1}) we have that for each $i\in E$, $n\geq 0$, and $h_n=(i_0,a_0,b_0,i_1,\ldots,a_{n-1},b_{n-1},i_n)\in H_n$,
		\begin{align}\label{h-ineq2}
				0 \leq &   {e^{ \theta (\varepsilon-\rho^*)}} \mathcal{L}^{ \pi_n(\cdot|h_n),\psi_h(\cdot|i_n)}h^* (i_n) -h^*(i_n)\notag\\
				=&   {e^{ \theta (\varepsilon-\rho^*)}}  \sum_{a\in A(i_n)}\sum_{b\in B(i_n)}\sum_{j\in E} e^{\theta c(i_n,a,b)}h^*(j)P(j|i_n,a,b)\pi_n(a|h_n)\psi_h(b|i_n)-h^*(i_n) \notag\\
				=&\mathbb{E}^{\pi,\psi_h}_{i}\big[ {e^{ \theta (\varepsilon-\rho^*)}}  e^{\theta c(X_n,A_n,B_n)}h^*(X_{n+1})-h^*(X_n)\big|Y_n=h_n\big],
			\end{align}
		where   $Y_{n }$ is defined in (\ref{Yn}).  
		Hence, it holds that
		\begin{align*}
				&   {e^{ (n+1)\theta (\varepsilon-\rho^*)}}  \mathbb{E}^{\pi,\psi_h}_{i}\big[e^{\theta \sum_{k=0}^{n }c(X_k,A_k,B_k)}h^*(X_{n+1})\big]\\
				=&  {e^{ n\theta (\varepsilon-\rho^*)}}  \mathbb{E}^{\pi,\psi_h}_{i}\Big[e^{\theta \sum_{k=0}^{n-1}c(X_k,A_k,B_k)}  \mathbb{E}^{\pi,\psi_h}_{i}\big[ {e^{ \theta (\varepsilon-\rho^*)}} e^{\theta  c(X_n,A_n,B_n)}h^*(X_{n+1})-h^*(X_{n})  \big|Y_n\big]\Big]\\
				&+  {e^{ n\theta (\varepsilon-\rho^*)}}  \mathbb{E}^{\pi,\psi_h}_{i} \big[e^{\theta \sum_{k=0}^{n-1}c(X_k,A_k,B_k)}h^*(X_{n})\big]\notag\\
				\geq & {e^{ n\theta (\varepsilon-\rho^*)}}  \mathbb{E}^{\pi,\psi_h}_{i} \big[e^{\theta \sum_{k=0}^{n-1}c(X_k,A_k,B_k)}h^*(X_{n})\big]  \qquad\qquad\qquad\qquad\qquad \forall n\in \mathbb{N},i\in E.
			\end{align*}
		Consequently, employing the    display above and an  induction argument, we obtain
		\begin{equation*}
			 {e^{ n\theta (\varepsilon-\rho^*)}} \EE^{\pi,\psi_h}_{i}\big[e^{\theta \sum_{k=0}^{n-1}c(X_k,A_k,B_k)}h^*(X_{n})\big]\geq   \EE^{\pi,\psi_h}_{i}\big[   h^*(X_{0})\big]\geq \min_{j\in E}h^*(j) \quad \forall n\in \mathbb{N},i\in E.
			\end{equation*}
	 From    Proposition \ref{pro-value} (b), we have    $\min_{j\in E}h^*(j)>0$. Hence,   
		\begin{equation}\label{h-ineq}
			 {\|h^*\|  }({\min_{j\in E}h^*(j)})^{-1} \EE^{\pi,\psi_h}_{i}\big[e^{\theta \sum_{k=0}^{n-1}c(X_k,A_k,B_k)} \big] \geq 	e^{n\theta (\rho^*-\varepsilon)}  \quad\forall n\geq 1, i\in E.
			\end{equation}
		Taking logarithm on  both sides, dividing by $\theta n$ and   letting $n\to \infty$ in (\ref{h-ineq}), we get
		$	J(i,\pi,\psi_h)\geq \rho^*-\varepsilon$ for all $i\in E$. Then, the arbitrariness of $\pi$ indicates that
	$$
				\inf_{\pi\in \Pi_1}	J(i,\pi,\psi_h)\geq \rho^*-\varepsilon \quad \forall i\in E, $$
which together with Definition \ref{defn-saddle} (b) and Proposition \ref{pro-value} (c) implies that $\psi_h$ is an $\varepsilon$-optimal policy for player 2.
		By analogous arguments as above we have that $\varphi_h$ is an $\varepsilon$-optimal policy for player 1. Therefore,  $(\varphi_h,\psi_h)$ is an $\varepsilon$-saddle point.
\end{proof}

According to Proposition \ref{var-sd}, to get an $\varepsilon$-saddle point, we only need to find   a function $h$  satisfying the condition $h\leq h^* \leq e^{\theta \varepsilon}h$,  and then obtain a mini-max selector of  $h$ by solving (\ref{LP}). Now we introduce some notation.
\begin{itemize}
\item [\rm i)]  Fix any state $i_*\in E$  in what follows.  Observing  that $h^*(i_*)>0$ (by  Proposition \ref{pro-value} (b)), and that the pair $(\rho^*, r h^*)$   satisfying $e^{\theta \rho^*}(rh^*)=\mathcal{L}(rh^*)$ for all $r> 0$,   without loss of generality,     assume that $h^*(i_*)=1$.

\item [\rm ii)]Denote 
\begin{align}\label{Mc}
M_c=\max_{(i,a,b)\in \mathbb{K}}c(i,a,b)-\min_{(i,a,b)\in \mathbb{K}}c(i,a,b).
\end{align} Obviously, if $M_c=0$, i.e., $c$ is a constant function,  any policy pair is a saddle point of the game. Hence, in what follows, we assume that $M_c>0$. 

\item [iii)] Given any $\varepsilon>0$ and $\theta\in (0, \frac{-\ln   (1-\gamma) }{|E|   M_{c}})$ ( $-ln 0:=\infty$), set
\begin{align}\label{l-var}
	&k_{\varepsilon}:= \begin{cases} \min\Big\{k\geq 1 \Big| k> \frac{\ln \gamma+  \ln   (1- (1-\gamma) e^{\theta |E|    M_{c}}  )+\ln  ( \exp(\frac{\theta \varepsilon}{2})-1)  - 2  \theta |E|   M_{c}}{  \theta   {|E|   M_{c}} + \ln  (1-\gamma) }\Big\},  &\gamma<1;\\
		1, \quad &\gamma=1;
	\end{cases}\\
	&N_{\varepsilon}:=   k_{\varepsilon} |E|, \label{N-var}
\end{align}
where $\gamma  $ is the irreducibility coefficient defined   in  (\ref{gamma}).
Under Assumption \ref{ir-ass} and the condition $ \theta\in (0, \frac{-\ln   (1-\gamma) }{|E|   M_{c}})$,
 both $k_\varepsilon$ and $N_\varepsilon$   are well defined.

\item [\rm iv)] For each $\varepsilon>0$, let  $\rho_\varepsilon$  denote the $\frac{\varepsilon}{2N_\varepsilon}$-approximation of the value computed by Algorithm \ref{algo1}, and therefore
$
	0\leq \rho_\varepsilon-\rho^*\leq {\varepsilon}/{2N_\varepsilon}.
$
Furthermore,  define a sequence  $\{U_n^\varepsilon,n\in \mathbb{N}\}$ of functions on $E$ by
\begin{align}\label{defn-U}
U_0^\varepsilon(i):=\delta_{i_*}(i), \quad	U_{n+1}^\varepsilon (i):=\delta_{i_*}(i)+e^{-\theta   \rho_\varepsilon  }(1-\delta_{{i_*}}(i))\mathcal{L}U_{n}^\varepsilon (i),  \quad i\in E.
\end{align}

\end{itemize}

 \begin{thm}\label{thm-sd}
	Assume that Assumption \ref{ir-ass} holds.   For any $\varepsilon>0$  and $ \theta\in (0, \frac{-\ln  (1-\gamma) }{|E|   M_{c}})$,     the following statements hold.
	\begin{itemize}
		\item [\rm (a)] The function $U^{\varepsilon}_{N_\varepsilon}$ satisfies the condition
		$
		U_{N_\varepsilon}^\varepsilon  \leq   h^* \leq   e^{\theta\varepsilon}U_{N_{\varepsilon}}^\varepsilon.
		$
		\item [\rm (b)] For each $i\in E$,  solve the linear program  and its dual  problem  in (\ref{LP}) with $h= U_{N_{\varepsilon}}^\varepsilon$, and denote their optimal solutions by $(w_\varepsilon(i),\varphi_\varepsilon(\cdot|i))$ and $(v_\varepsilon(i),\psi_\varepsilon(\cdot|i))$, respectively. Then, $(\varphi_\varepsilon,\psi_\varepsilon)$ is an $\varepsilon$-saddle point. 
	\end{itemize}
\end{thm}
\begin{proof}
	(a) We    establish  part (a) by steps.
	
	{\bf Step 1:} We show   that                                                                                                                                                                                                                                                                                                                                                                                                                                                                                                                                                                                                                                                                                                                                                                                                                                                                                                                                                                                                                                                                  
	\begin{equation}\label{thm-sd-1}
		U_n^{\varepsilon }  \leq  h^*  \quad \forall n\in \mathbb{N}.
	\end{equation}
	It follows from $h^*({i_*})=1$ and $h^*\in \mathcal{M}_+$ that $U^\varepsilon_0=\delta_{i_*}\leq h^*$. Suppose that  $U^\varepsilon_n\leq h^*$ is true for some $n\geq 0$, then using $
	\rho^*\leq \rho_\varepsilon
	$, (\ref{defn-U}), and Proposition \ref{pro-value} we have that
	\begin{align}
		U_{n+1}^{\varepsilon }(i)=e^{-\theta \rho_\varepsilon}\mathcal{L}U_n^\varepsilon(i)\leq e^{-\theta \rho^*}\mathcal{L}h^*(i)=h^*(i) \quad \forall i\in E\setminus \{{i_*}\},
	\end{align}
	which, together with $h^*({i_*})=	U_{n+1}^{\varepsilon }({i_*})=1$, implies that $U_{n+1}^{\varepsilon } \leq  h^*$. Therefore, by induction we obtain (\ref{thm-sd-1}).
	
	{\bf Step 2: }   We prove   that
	\begin{align}\label{Gineq2-1}
		 h^*(i)\leq U_{N_\varepsilon}^*(i)+ \gamma	e^{-\theta|E|  M_{c}}  (e^{\frac{\theta\varepsilon}{2}}-1) \quad \forall i\in E.
	\end{align}
where  $U_0^*  :=\delta_{i_*}=U_0^\varepsilon $ and
\begin{align}\label{defn-UU}
 U_{n+1}^*  (i) :=\delta_{i_*}(i)+e^{-\theta   \rho^*  }(1-\delta_{{i_*}}(i))\mathcal{L}U_{n}^*  (i),  \quad i\in E.
\end{align}
Note that $U_1^*\geq \delta_{i^*}=U_0^*$. Moreover, if $U_{n+1}^*\geq U_{n}^*$ for some $n\in \mathbb{N}$, then by (\ref{L-mono}) we have
$$
U_{n+2}^*(i)=	e^{-\theta   \rho^*  }(1-\delta_{{i_*}}(i))\mathcal{L}U_{n+1}^*  (i) \geq 	e^{-\theta   \rho^*  }(1-\delta_{{i_*}}(i))\mathcal{L}U_{n}^*  (i) =U_{n+1}(i) \quad \forall i\in E\setminus \{i^*\},
$$
which together with (\ref{defn-UU}) implies that $U_{n+2}^* \geq U_{n+1}^*$. Hence, by induction  we  deduce  
\begin{align}\label{U-mono}
	U^*_{n+1}\geq U^*_n \quad \forall n\in \mathbb{N}.
\end{align}
Next,  we claim that for all $ m,n,k\in \mathbb{N}$, it holds that
\begin{align}\label{U-ineq}
   {U}^*_{ m+k}(i)\leq {U}^*_{n+k}(i)+ \|{U}^*_{m}-{U}^*_{n}\| e^{\theta k  M_{c}} V_k^{i_*}(i) \quad \forall i\in E\setminus \{i^*\},
\end{align}
where $V_k^ {i_*} $ is defined in (\ref{Vxy}).
Let $m$ and $n$ be fixed but arbitrary. The proof of (\ref{U-ineq}) proceeds by induction on $k$. 
Since    $V_{0}^{i_*}\equiv1 $ (by (\ref{Vxy})) and ${U}^*_{ m } \leq {U}^*_{ n  }+ \|{U}^*_{m}-{U}^*_{n}\|$,  
    (\ref{U-ineq}) holds for $k=0$.  For the induction step, assume that  (\ref{U-ineq}) holds for some $k\geq 0$. Then, for any $i\in E\setminus \{i^*\}$, $a\in A(i)$, and $b\in B(i)$,
    using   $U_{n+k}^*(i^*)=U_{m+k}^*(i^*)=1$ we obtain 
    \begin{align*}
    &\sum_{j\in E}P(j|i,a,b)U_{m+k}^*(j)\\
    =&  P(i^*|i,a,b)U_{n+k}^*(i^*)+\sum_{j\in E\setminus \{i^*\}}P(j|i,a,b)U^*_{m+k}(j)   \notag\\
    \leq &  P(i^*|i,a,b)U_{n+k}^*(i^*)+\sum_{j\in E\setminus \{i^*\}}P(j|i,a,b)({U}^*_{n+k}(j)+ \|{U}^*_{m}-{U}^*_{n}\| e^{\theta k  M_{c}} V_k^{i_*}(j)) \notag\\
   =& \sum_{j\in E }P(j|i,a,b){U}^*_{n+k}(j)+\|{U}^*_{m}-{U}^*_{n}\| e^{\theta k  M_{c}} \sum_{j\in E\setminus \{i^*\}}P(j|i,a,b)  V_k^{i_*}(j)  \notag\\
   \leq& \sum_{j\in E }P(j|i,a,b){U}^*_{n+k}(j)+\|{U}^*_{m}-{U}^*_{n}\| e^{\theta k  M_{c}} V_{k+1}^{i_*}(i),
    \end{align*}
   where the first and last inequalities are due to the   induction hypothesis and (\ref{Vxy}), respectively. This, together with (\ref{L}),   (\ref{L-uv}), and (\ref{defn-UU}), implies that for each $i\in E\setminus \{i^*\}$, 
    \begin{align}\label{U-mnk}
	 {U}^*_{m+k+1}(i) 
	=& e^{-\theta \rho^*}\inf_{\mu\in \mathcal{P}(A(i))}\sup_{\nu\in \mathcal{P}(B(i))}\Big\{\sum_{a\in A(i)}\sum_{b\in B(i)}\mu(a)\nu(b)e^{\theta c(i,a,b)}\sum_{j\in E}P(j|i,a,b)	{U}^*_{m+k }(j) \Big\} \notag\\
	\leq &e^{-\theta \rho^*}\inf_{\mu\in \mathcal{P}(A(i))}\sup_{\nu\in \mathcal{P}(B(i))}\Big\{\sum_{a\in A(i)}\sum_{b\in B(i)}\mu(a)\nu(b)e^{\theta c(i,a,b)} \sum_{j\in E }P(j|i,a,b){U}^*_{n+k}(j)\notag\\
	&\qquad\qquad  \qquad \qquad\qquad +\sum_{a\in A(i)}\sum_{b\in B(i)}\mu(a)\nu(b)e^{\theta c(i,a,b)}\|{U}^*_{m}-{U}^*_{n}\| e^{\theta k  M_{c}} V_{k+1}^{i_*}(i) \Big\} \notag\\
	\leq &\mathcal{L} {U}^*_{n+k}(i)+\|{U}^*_{m}-{U}^*_{n}\| e^{\theta k  M_{c}} V_{k+1}^{i_*}(i)  e^{\theta (||c||-\rho^*)},
	\end{align}
with $||c||=\max_{(i,a,b)\in\mathbb{K}}c(i,a,b)$. 
Since (\ref{d-rs}) implies  
$ 
 \min_{(i,a,b)\in \mathbb{K}}c(i,a,b) \leq J(i,\pi,\sigma) \leq  ||c||$ 
for all $i\in E$  and $(\pi,\sigma)\in \Pi_1\times \Pi_2$,   we have that
\begin{align}\label{rho-ineq}
	 \min_{(i,a,b)\in \mathbb{K}}c(i,a,b) \leq  \rho^* \leq  ||c||,
\end{align} 
 which together with (\ref{Mc}) yields 
$
  ||c|| - \rho^*\leq ||c||-\min_{(i,a,b)\in \mathbb{K}}c(i,a,b)=M_c.
$
Combining this and (\ref{U-mnk}) we obtain 
\begin{align*}
	 {U}^*_{m+k+1}(i)  \leq \mathcal{L} {U}^*_{n+k}(i)+\|{U}^*_{m}-{U}^*_{n}\| e^{\theta {k+1}  M_{c}} V_{k+1}^{i_*}(i)  \quad \forall  i\in E\setminus \{i^*\}.
\end{align*}
  Hence,  by induction we establish this claim.
  Then,  using (\ref{defn-UU})-(\ref{U-ineq})    we have that 
 \begin{align*}
\|U_{(n+1)|E|}^*-U_{n|E|}^*\| 
=&\max_{i\in E\setminus\{i^*\}}(U_{(n+1)|E|}^*(i)-U_{n|E|}^*(i)) \\
\leq & e^{\theta|E|M_c}\Big(\max_{i\in E}V_{|E|}^{i_*} (i)\Big) \|U_{n|E|}^*-U_{(n-1)|E|}^*\| 
=  \beta_\theta \|U_{n|E|}^*-U_{(n-1)|E|}^*\|,
 \end{align*}
for all $n\geq 1$, with 
$ 
 	\beta_{\theta}:=    e^{ \theta|E|    M_{c}}\max_{i\in E}V_{|E|}^{i_*} (i).
$ 
This, together with   (\ref{N-var}), implies that
 \begin{align}\label{U-z1}
 \|U_{N_\varepsilon+|E|}^*-U_{N_\varepsilon}^*\|  \leq \beta_\theta  \|U_{N_\varepsilon}^*-U_{N_\varepsilon-|E|}^*\| \leq \cdots  \leq \beta^{k_\varepsilon}_\theta \|U^*_{|E|}-U^*_{0}\| \leq  \beta^{k_\varepsilon}_\theta \|U^*_{|E|} \|.
 \end{align}
By $||U_0^*||=1$, we see that $\|U^*_{ n } \|\leq e^{n\theta  M_c}$ holds when $n=0$.  If for some $n\geq 0$, it holds that $\|U^*_{ n } \|\leq e^{n\theta   M_c}$, then     by (\ref{defn-UU})  we deduce that
  \begin{align}\label{U1-ineq2}
 	\|U^*_{ n+1 } \| =\max\{1, \max_{i\in E\setminus \{i\}}e^{-\theta\rho^*}\mathcal{L}U_{n }^*(i)\}\leq  \max\{1,  ||U_{n }^*||  e^{\theta (\|c\|-\rho^*)}\} \leq e^{ (n+1)\theta   M_c},
 \end{align}
where the first  and last inequalities follow from  (\ref{L-1})  and    $\|c\|-\rho^*\leq M_c$, respectively.
 Thus,  by induction we get  $\|U^*_{ n } \|\leq e^{n\theta   M_c}$ for all $n\in \mathbb{N}$, which together with (\ref{U-z1}) gives 
  \begin{align}\label{U-z2}
 	\|U_{N_\varepsilon+|E|}^*-U_{N_\varepsilon}^*\|\leq  \beta^{k_\varepsilon}_\theta e^{ |E| \theta   M_c}.
 \end{align}
Combining the definition of $\beta_{\theta}$ and Proposition \ref{prop-N}, we obtain that
$ 
\beta_{\theta}\leq  e^{ \theta|E|    M_{c}}(1-\gamma),
$ 
which together with  $ \theta\in (0, \frac{-\ln   (1-\gamma) }{|E|   M_{c}})$ implies   $\beta_\theta<1$. Then,
 by (\ref{l-var}) and (\ref{U-z2}), we have
\begin{align}\label{U-var1}
 \|U_{N_\varepsilon+|E|}^*-U_{N_\varepsilon}^*\|  \leq  \beta^{k_\varepsilon}_\theta e^{\theta |E|M_c}  \leq \gamma e^{-\theta|E|  M_{c}}(1-\beta_\theta) (e^{\frac{\theta\varepsilon}{2}}-1).
\end{align}
On the other hand,  it follows from $h^*({i_*})=1$ and Proposition \ref{pro-value} (b)  that
   \begin{align}\label{h-SE}
     h^*(i)=\delta_{i_*}(i)+e^{-\theta \rho^*}(1-\delta_{i_*}(i))\mathcal{L}h^*(i) \quad \forall i\in E.
 \end{align}                                                                            Using  the arguments similar to the proof of  (\ref{U-ineq}), we also have that
\begin{align}\label{h-U-ineq}
h^*(i)\leq   {U}^*_{n+k} (i)+\|h^*- {U}_{n }^*\|  e^{\theta k  M_{c}}(1-\delta_{i_*}(i))V_k^{i_*}(i) \quad \forall  i\in E\setminus \{i^*\}, k,n\in \mathbb{N}.
\end{align}
Combining   (\ref{U-var1})  and (\ref{h-U-ineq}), we get that
 \begin{align*}
 \|h^*-U_{N_\varepsilon}^*\|\leq & \|h^*-U_{N_\varepsilon+|E|}^*\|+\|U_{N_\varepsilon+|E|}^*-U_{N_\varepsilon}^*\| \notag\\
 \leq  & \beta_\theta \|h^*-U_{N_\varepsilon}^*\|+ \gamma	e^{-\theta|E|  M_{c}}(1-\beta_\theta) (e^{\frac{\theta\varepsilon}{2}}-1).
 \end{align*}
This gives that $h^*(i)\leq U_{N_\varepsilon}^*(i)+ \gamma	e^{-\theta|E|  M_{c}}  (e^{\frac{\theta\varepsilon}{2}}-1) $    for all $i\in E$,   i.e., (\ref{Gineq2-1})  is true.

 {\bf Step 3}:
Now, we show that for each $i\in E$, 
 	\begin{equation}\label{Gineq2}
 	h^*(i) \leq e^{{\theta \varepsilon}/{2}}  {U}_{N_{\varepsilon}}^*(i).
 \end{equation}
First, we claim that   for each $k\geq 0$,
 \begin{align}\label{Tz-1}
 	e^{-k\theta  M_{c} }(1-V_k^{i_*}(i))	\leq  U^*_k(i) \quad \forall  i\in E\setminus \{{i_*}\}.
 \end{align}
By   (\ref{Vxy}) and $U^*_0\in \mathcal{M}_+$,  (\ref{Tz-1}) holds for $k=0$.  For the induction
 step,   assume that (\ref{Tz-1}) is true for some $k\geq 0$.   Then, for each $i\in E\setminus \{{i_*}\}$,  we obtain that
 \begin{align*}
  U^*_{k+1}(i)=  &e^{-\theta \rho^*}\mathcal{L}  U^*_{k }(i) \\
 	\geq &e^{-\theta \rho^*} \min_{a\in  A(i) }\min_{b\in  B(i) } e^{\theta c(i,a,b)} \big(\sum_{j\neq {i_*}} e^{- k\theta  M_{c} }(1-V_k^{i_*}(j))P(j|i,a,b)+P({i_*}|i,a,b) \big)   \\
 	\geq &e^{- k\theta  M_{c} }e^{-\theta  \rho^*  } \min_{a\in  A(i) }\min_{b\in  B(i) } e^{\theta c(i,a,b)} \big(1-\sum_{j\neq {i_*}}  V_k^{i_*}(j) P(j|i,a,b)  \big)   \\
 	\geq &e^{- k\theta  M_{c} }  \min_{a\in  A(i) }\min_{b\in  B(i) } e^{-\theta (\rho^*- c(i,a,b))}  \times\min_{a\in  A(i) }\min_{b\in  B(i) }(1-\sum_{j\neq {i_*}}  V_k^{i_*}(j) P(j|i,a,b)  \big)   \\
 	\geq &e^{- (k+1) \theta  M_{c} }(1- V_{k+1}^{i_*}(i)),    
 \end{align*}
 where  the last inequality is due to (\ref{Vxy}) and $\rho^*-\min_{(i,a,b)\in \mathbb{K}}c(i,a,b)\leq M_c$.    Therefore,  by induction (\ref{Tz-1}) holds for all $k\geq 0$.
Noting that $N_\varepsilon=k_\varepsilon|E|\geq |E|$,  we obtain by (\ref{U-mono}), (\ref{Tz-1}) and Proposition \ref{prop-N} (a) that
 \begin{align*}
 U_{N_\varepsilon}^*(i)   \geq 	U_{|E|}^*(i) \geq e^{-\theta|E|  M_{c}} (1-V_{|E|}^{i_*}(i))
  \geq \gamma e^{-\theta|E|  M_{c}}   \quad \forall i\in E\setminus \{{i_*}\}.
 \end{align*}
 This,  together with (\ref{Gineq2-1})  and $h^*(i^*)= U_{N_\varepsilon}^*(i^*)=1$, completes Step 3.

 {\bf Step 4:} Using  $
 0\leq \rho_\varepsilon-\rho^*\leq {\varepsilon}/{2N_\varepsilon}
 $, (\ref{defn-U}), and (\ref{defn-UU}) we obtain by induction that
 \begin{align}
 	U^*_n \leq e^{ {n\theta \varepsilon}/{2N_\varepsilon}}U^\varepsilon_n \quad \forall n\geq 0,
 \end{align}
 which together with (\ref{Gineq2}) yields that
$
 h^* \leq e^{ {\theta  \varepsilon}/{2 }}U^*_{N_\varepsilon}\leq e^{ {N_\varepsilon\theta \varepsilon}/{2N_\varepsilon}}  e^{ {  \theta\varepsilon}/{2 }}U^{\varepsilon}_{N_\varepsilon} =e^{\theta \varepsilon}U^{\varepsilon}_{N_\varepsilon}.
$
This  and (\ref{thm-sd-1}) lead to part (a).

 (b) It directly follows from part (a) and Proposition \ref{var-sd}.
\end{proof}

\begin{cor}\label{sd-cor}
	If we take  $i_*:=\arg\min\limits_{j\in E}\{\max\limits_{i\in E}V_{|E|}^{j}(i)\}$ and  replace $\gamma$ by $ \eta:=1-\max_{i\in E}V_{|E|}^{i_*}(i) $ in (\ref{l-var}), then for any $\theta \in (0, \frac{-\ln (1-\eta)}{|E| M_c})$, the results in Theorem \ref{thm-sd} are still valid.
\end{cor}
\begin{proof}
	By replacing $ \gamma$ with $\eta$ in the proof of Theorem \ref{thm-sd}, we get the corollary. 
\end{proof}

Based on Corollary \ref{sd-cor}, under Assumption  \ref{ir-ass}  and the condition  $ \theta\in (0, \frac{-\ln  (1- \eta) }{|E|   M_{c}})$, we now propose an algorithm  to compute    $\varepsilon$-saddle points.

\begin{algorithm}
	\caption{(An algorithm of $\varepsilon$-saddle points.)}
	\label{algo2}
	\begin{algorithmic}
 	\STATE {\bf Input}: The data of the game and  an accuracy $\varepsilon>0$.
 \STATE {\bf  Step 1} (Initialization): Set $i_*:=\arg\min\limits_{j\in E}\{\max\limits_{i\in E}V_{|E|}^{j}(i)\}$,    $U_0^\varepsilon=\delta_{i_*}$,  $n=1$, and
			\begin{align*}
				&k_{\varepsilon} =\begin{cases} \min\Big\{k\geq  1 \Big| k> \frac{\ln \eta+ \ln   (1- (1-\eta) e^{\theta|E|     M_{c}} ) +\ln  ( \exp(\frac{\theta \varepsilon}{2})-1) - 2  \theta |E|   M_{c}}{   \theta  |E|    M_{c}  + \ln   (1-\eta) }\Big\},  &\eta<1;\\
					1, \quad &\eta=1;
				\end{cases}
			\end{align*}
			and $N_{\varepsilon} =   k_{\varepsilon} |E|$,
			where  $ \eta:=1-\max_{i\in E}V_{|E|}^{i_*}(i) $.

\STATE	  {\bf  Step 2} (Approximate the value):
	 	Input the accuracy $\frac{\varepsilon}{2N_{\varepsilon}}$ in Algorithm \ref{algo1}.  Then an $\frac{\varepsilon}{2N_{\varepsilon}}$-approximation $ {\rho}_\varepsilon $ of the value   is output, and ${\rho}_\varepsilon $ satisfies $0\leq {\rho}_\varepsilon-\rho^*\leq
	 	\frac{\varepsilon}{2N_{\varepsilon}}$.

\STATE {\bf Step 3} (Iteration):
For each $i\in E$,  compute $\mathcal{L}U_{n-1}^\varepsilon(i)$ by solving (\ref{LP}) with $h=U_{n-1}^\varepsilon$, and let
	\begin{equation*}
		U_{n}^\varepsilon(i)=\delta_{i_*}(i)+ e^{-\theta  {\rho}_\varepsilon}(1-\delta_{i_*}(i))\mathcal{L}U_{n-1}^\varepsilon(i).
	\end{equation*}

\STATE {\bf Step 4}: If $n<N_\varepsilon$,  go to Step 3 by replacing $n$ with $n + 1$. Otherwise,  for each $i\in E$, solve the   linear program  problem    and its dual problem in ( \ref{LP}) with $h=U_{n }^\varepsilon$
 and denote    their optimal solutions, respectively, by $(w(i),\varphi_\varepsilon(\cdot|i))$ and $(v(i),\psi_\varepsilon(\cdot|i))$.
\STATE	{ \bf  Output}: An   $\varepsilon$-saddle point $(\varphi_{\varepsilon},\psi_{\varepsilon})$.

	\end{algorithmic}
\end{algorithm}

\section{ An  example of energy management in smart grids}
In this section, we introduce an    example of energy management in smart grids. 

\begin{exm}
There is a prosumer {\bf A} equipped with a  storage unit of maximum capacity $N_s$ and   renewable energy generators, such as   wind turbines,   tidal turbines, or   solar panels. 
The stored charge in the storage unit    may be reduced due to self-discharge resulting from  internal chemical reactions \cite{WW00}.
Additionally, since  the amount of energy generated by   renewable energy generators depends on   natural environments, the energy harvested from renewable resources is random.  Denote by $G_n$ ($n\in \mathbb{N}$) the {\it effective} generated energy during   period $[n, n+1)$,  which is the energy harvested from renewable resources minus the storage self-discharge. As in  \cite{E24,ESMP18}, the distribution of ${G}_{n}$ is assumed to be independent of time and is denoted by $Q$.  
Prosumer {\bf A} can  also purchase    energy   from a utility company   and consume energy. 
When the  storage level at the beginning of    period $[n, n+1)$  is $i_n\in \{0,1,\ldots,N_s\}$,  prosumer {\bf A}  may  purchase $b_n^{p}$ $(0\leq b^{p}_n\leq   N_{p}) $ units of energy,   and    consume $b_n^{c}$ $( 0\leq b^{c}_n\leq  (i_n+b_n^{p})\wedge N_{c} ) $ units of energy, where $N_{p}$ and  $N_{c}$ represent the maximum purchased energy and the maximum consumable energy, respectively. 
  Then,  the   storage level at   the beginning of    period $[n+1, n+2)$  will be  $\min \{N_s, \max\{0,G_n+i_n+ b_n^{p}-b_n^{c}\}\}$. 
  
On the other hand,  since the energy price  set by     the utility company   depends on  total demands of the electricity market,   the cost of purchasing   energy for  prosumer {\bf A} depends on the   energy demand $a\in \{0,1,\ldots, {M}\}$   of the other prosumers, where    $M$  is the maximum purchased energy  by the other prosumers.  
When the   energy demands  of prosumer {\bf A} and  the other prosumers  are $b^p$ and $a$, respectively,    the cost of  prosumer {\bf A}  is denoted by $C(a,b ^p)$.
Moreover,  let $R(b^c)$   represent  the profit    from consuming  $b^c$ units of energy.
Then the payoff of  prosumer {\bf A}   is $R(b^c) -C(a,b^{p} )$. 
Since the    energy demand of the other prosumers is  unknown for  prosumer {\bf A}, 
we suppose that  Prosumer {\bf A} wants to  find an optimal policy in a zero-sum game under the risk-sensitive average criterion. 
\end{exm}

We now formulate this energy management problem as a risk-sensitive stochastic game, where player 2 and player 1 are prosumer {\bf A} and the other prosumers  in the electricity market, respectively.
 Let $E=\{0,1,\ldots,N_s\}$ be the state space.  The action spaces of player 1 and player 2 are $A=\{0,1,\ldots, {M}\}$ and $B=\{0,1,\ldots,N_{p}\}\times \{0,1,\ldots,N_{c}\}$,  respectively. Given any $i\in E$, we have $A(i)=A$ and
 $  B(i)=\{(b^{p}, b^c): b^p\in \{0,1,\ldots,N_{p}\}, b^c\in \{0,1,\ldots, (i+b^p)\wedge N_{s} \}\}.$ 
Moreover, the stochastic kernel $P$ and the   reward function $c$ for player 2 are given by: for all $i,j\in E$, $a\in A(i)$ and $b=(b^p, b^c)\in B(i)$
\begin{align}
&P(j|i,a,b)=\begin{cases}
\sum_{k=-\infty}^{  -(i+b^p-b^c)}Q(k) & j=0,\\
 Q(j-(i+b^p-b^c)) \   & 0<j<N_s, \label{exm-P}\\
 \sum_{k=N_s-(i+b^p-b^c)}^{\infty}Q(k)  \  &j=N_s,
\end{cases} 
&c(i,a,b)=R(b^c) -C(a,b^p),
\end{align}

To  ensure the existence of the value and a saddle point, and compute $\varepsilon$-saddle points for the   game, as in \cite{ESMP18}  we propose the following assumption.
\begin{ass}\label{exm-ass}
 $Q(k)>0$ for all $k\in \{-N_{c},-N_{c}+1,\ldots,0,1,\ldots, N_s\}$. 
\end{ass}
\begin{prop}
Under Assumption \ref{exm-ass},  the game has the value and   a saddle point.
\end{prop}
\begin{proof}
By (\ref{exm-P}), it holds that  $\mathbb{P}_i^{\pi,\sigma}(X_1=j)\geq \min\limits_{a\in A(i), b\in B(i)}P(j|i,a,b)\geq  \min\limits_{-N_c\leq k\leq N_s}Q(k)>0$  for all $i,j\in E$ and $(\pi,\sigma)\in \Pi_1\times \Pi_2$, which implies Assumption \ref{ir-ass}. Therefore, by Proposition  \ref{pro-value}  we get the desired result.  
\end{proof}

Next,     we take numerical calculations and compute an $\varepsilon$-approximation of the value and an $\varepsilon$-saddle point for this  game.  Without loss of generality, let $N_s=2$,  $N_{c}=3$, $N_{p}=2$, ${M}=2$, $\theta=0.01$, and
\begin{align}\label{P-G}
&Q(k) =\int_{k}^{k+1} \frac{1}{2\sqrt{2\pi}}\exp( -{(x-1)^2}/8)\d x, \qquad \quad \, k=0,-1,1,-2,2,\ldots;\\
&R(b^{c })=\ln(b^c+0.4)-\ln0.4, \qquad\qquad\qquad \qquad b^c\in \{0,1,2,3\}; \notag \\
& C(a,b^p)=  b^p \times {    (a+b^p )}/{10}  +    {\bf 1}_{(a,\infty)}(b^p )/4 , \qquad a\in \{0,1,2\},b^p\in\{0,1,2\} \notag, 
\end{align}
where   $  (a+b^p )/10$ is the energy price  based on all players’ demand, and $ {\bf 1}_{(a,\infty)}(b^p )/4$ (with  ${\bf 1}_{(a,\infty)}(b^p )=1$ if $b^p>a$, and ${\bf 1}_{(a,\infty)}(b^p )=0$ if $b^p\leq a$)    means that if    player 2 purchases more energy than player 1,  player  2 needs to pay   an additional fee of $1/4$.

It can be verified by (\ref{P-G}) that  Assumption  \ref {exm-ass} holds, so Assumption  \ref{ir-ass} holds.      Moreover, by (\ref{Vxy}), (\ref{exm-P}) and (\ref{P-G}), a directive calculation gives   $i^*=\arg\min\limits_{j\in E}\max\limits_{i\in E}V^{j}_{|E|}(i)=2$ and $
\eta= 1-\max_{i\in E}V^{i^*}_{|E|}(i)=   0.6694$     with $|E|=3$. Then,    we obtain $\frac{-\ln (1-\eta)}{|E|M_c}=   0.1322$  and $\theta=0.01\in (0,\frac{-\ln (1-\eta)}{|E|M_c})$,  where  $M_c=\max\limits_{(i,a,b)\in \mathbb{K}}c(i,a,b)-\min\limits_{(i,a,b)\in \mathbb{K}}c(i,a,b)=     2.7901$. This, together with Theorems \ref{al-value} and Corollary \ref{sd-cor},  implies that both Algorithms  \ref{algo1} and   \ref{algo2} are valid for this example.

Take $\varepsilon=0.05$. Using the data above and Algorithms \ref{algo1} and \ref{algo2}, we compute an $8.3\times 10^{-4}$-approximation $\rho_{\varepsilon}=  1.3217$ of the value  and a  $0.05$-saddle point $(\varphi_{\varepsilon},\psi_{\varepsilon})$ of the game, where   $\varphi_\varepsilon$ and $\psi_\varepsilon$ are shown in Table \ref{DTMG_P1} and Table \ref{DTMG_P2}, respectively, 

	\begin{table}[htp]
 	\centering
  	\fontsize{ 4}{8}\selectfont
	 		\linespread{1.5}
 		\small
 	\begin{tabular}{|c|c|c|c| }
	 			\hline
	 			 \diagbox[width=4.5em, height=2em]{   $i$}{  $a$}& 0&1&2   \\	\hline
		 		 		$0$ &0.0000  &0.2599   &0.7401\\ \hline
		 			 	$1$ &0.7584&0.0000&0.2416 \\ \hline
		 			 		$2$ &1.0000  &0.0000&0.0000 \\ \hline
		 		\end{tabular}
		 			\caption{$0.05$-optimal policy $\varphi_{\varepsilon}$ of player 1}\label{DTMG_P1}
	 	\end{table}
	 
 	\begin{table}[htp]
 		\centering
 		\fontsize{4}{8}\selectfont
 		\linespread{1.5}
 		\small
 		\begin{tabular}{|c|c|c|c|c|c|c|c|c|}
 			\hline
 			\diagbox[width=7em,height=3em]{ $i$}{ $ (b^{p},b^{c})$} &(0,2)&(1,1)&(1,2) & (2,2) & (2,3) &\makecell{other \vspace{-0.5em}\\actions} 	 \\	\hline
 			$0$ &0.0000 &0.3333 &0.0000&0.6667&0.0000  &0.0000\\ \hline
	 			$1$ &0.0000 &0.0000&0.7499&0.0000&0.2501 &0.0000\\ \hline
 			$2$ &1.0000 &0.0000&0.0000&0.0000&0.0000&0.0000\\ \hline		
 		\end{tabular}	
 		\caption{$0.05$-optimal policy $\psi_{\varepsilon}$ of player 2}\label{DTMG_P2}
 	\end{table}
 
\bibliographystyle{alpha}

\end{document}